\documentclass{llncs}

\usepackage{amssymb,stmaryrd,mathrsfs,dsfont,mathtools,enumitem}
\usepackage[ruled,linesnumbered]{algorithm2e}
\usepackage{algpseudocode}

\usepackage{pgfplots}
\usepackage{xfrac} 
\pgfplotsset{compat=1.15}
\usetikzlibrary{arrows}
\usepackage{tikz}
\tikzstyle{vertex}=[circle, draw, inner sep=0pt, minimum size=3pt]
\newcommand{\vertex}{\node[vertex]}

\begin{document}
\frontmatter          % for the preliminaries
\pagestyle{headings}  % switches on printing of running heads

\mainmatter              % start of the contributions

\title{Representation Number of\\ Word-Representable Split Graphs}

\titlerunning{Title}  % abbreviated title (for running head)

\author{Tithi Dwary \and Khyodeno Mozhui  \and %\inst{1} 
K. V. Krishna} %\inst{2}

\authorrunning{Tithi Dwary \and Khyodeno Mozhui \and K. V. Krishna } % abbreviated author list (for running head)

\institute{Indian Institute of Technology Guwahati, India\\
	\email{tithi.dwary@iitg.ac.in};\;\;\; 
	\email{k.mozhui@iitg.ac.in};\;\;\;
	\email{kvk@iitg.ac.in}}

\maketitle              % typeset the title of the contribution

\begin{abstract}
	A split graph is a graph whose vertex set can be partitioned into a clique and an independent set. The word-representability of split graphs was studied in a series of papers in the literature, and the class of word-representable split graphs was characterized through semi-transitive orientation. Nonetheless, the representation number of this class of graphs is still not known. In general, determining the representation number of a word-representable graph is an NP-complete problem. 
	In this work, through an algorithmic procedure, we show that the representation number of the class of word-representable split graphs is at most three. Further, we characterize the class of word-representable split graphs as well as the class of split comparability graphs which have representation number exactly three.	
	
\end{abstract}


\section{Introduction}
A word over a finite set of letters is a finite sequence which is written by juxtaposing the letters of the sequence. A subword $u$ of a word $w$, denoted by $u \ll w$, is defined as a subsequence of the sequence $w$. For instance, $aabccb \ll acabbccb$. Let $w$ be a word over a set $X$, and $A \subseteq X$. Then, $w|_A$ denotes the subword of $w$ that precisely consists of all occurrences of the letters of $A$. For example, if $w = acabbccb$, then $w|_{\{a, b\}} = aabbb$. For a word $w$, if $w|_{\{a, b\}}$ is of the form $abab \cdots$ or $baba \cdots$, which can be of even or odd length, we say the letters $a$ and $b$ alternate in $w$; otherwise we say $a$ and $b$ do not alternate in $w$. A $k$-uniform word is a word in which every letter occurs exactly $k$ times. For a word $w$, we write $w^R$ to denote its reversal.

A simple graph $G = (V, E)$ is called a word-representable graph, if there exists a word $w$ over $V$ such that for all $a, b \in V$, $\overline{ab} \in E$ if and only if $a$ and $b$ alternate in $w$. Although, the class of word-representable graphs was first introduced in the context of  Perkin semigroups \cite{perkinsemigroup}, this class of graphs received attention of many authors due to its combinatorial properties. It was proved that a graph is word-representable if and only if it admits a semi-transitive orientation \cite{Halldorsson_2016}. The class of word-representable graphs includes several important classes of graphs such as comparability graphs, circle graphs, $3$-colorable graphs and parity graphs. One may refer to the monograph \cite{words&graphs} for a complete introduction to the theory of word-representable graphs.

A word-representable graph $G$ is said to be $k$-word-representable if there is a $k$-uniform word representing it. In \cite{MR2467435}, It was proved that every word-representable graph is $k$-word-representable, for some $k$. The representation number of a word-representable graph $G$, denoted by $\mathcal{R}(G)$, is defined as the smallest number $k$ such that $G$ is $k$-word-representable. A word-representable graph $G$ is said to be permutationally representable if there is a word of the form $p_1p_2 \cdots p_k$ representing $G$, where each $p_i$ is a permutation on the vertices of $G$; in this case $G$ is called a permutationally $k$-representable graph. The permutation-representation number (in short \textit{prn}) of $G$, denoted by $\mathcal{R}^p(G)$, is the smallest number $k$ such that $G$ is permutationally $k$-representable. It was shown in \cite{perkinsemigroup} that a graph is permutationally representable if and only if it is a comparability graph - a graph which admits a transitive orientation. Further, if $G$ is a comparability graph, then $\mathcal{R}^p(G)$ is precisely the dimension of an induced poset of $G$ (cf. \cite{khyodeno2}). It is clear that for a comparability graph $G$, $\mathcal{R}(G) \le \mathcal{R}^p(G)$. 

The class of graphs with representation number at most two is characterized as the class of circle graphs \cite{Hallsorsson_2011} and the class of graphs with \textit{prn} at most two is the class of permutation graphs \cite{Gallaipaper}. In general, the problems of determining the representation number of a word-representable graph, and the \textit{prn} of a comparability graph are computationally hard \cite{Hallsorsson_2011,yannakakis1982complexity}. However, it was established that Cartesian product of $K_m$ and $K_2$ \cite{MR4023050}, prisms \cite{Kitaev_2013}, the Petersen graph \cite{words&graphs}, and book graphs \cite{MR4023050,mozhui2023} have representation number at most three. Further, in \cite{rep_crown}, the representation number of a crown graph was determined.

A split graph is a graph in which the vertex set can be partitioned into a clique and an independent set.  In the context of word-representable graphs, word-representability of split graphs was studied in a series of papers (cf. \cite{Chen_2022,Iamthong_2022,Iamthong_2023,Kitaev_2021,Kitaev_2024}), and the class of split graphs which are word-representable, called as word-representable split graphs, was characterized in \cite{Kitaev_2021} using semi-transitive orientation. In Section \ref{split_intro}, along with a detailed information about split graphs, we reconcile relevant results known for the class of split graphs restricted to circle graphs and comparability graphs.  While the \textit{prn} of a split comparability graph was known to be at most three (cf. \cite{split_orders2004}), the representation number of the class of word-representable split graphs is still unknown. 

In this paper, we show that the representation number of a word-representable split graph is at most three.  Indeed, based on the characterization of word-representable split graphs given in \cite{Kitaev_2021,Kitaev_2024}, we devise an algorithmic procedure to construct a 3-uniform word which represents a given word-representable split graph.  Additionally, we characterize the class of word-representable split graphs as well as the class of split comparability graphs which have representation number exactly three.

\section{Split Graphs}\label{split_intro}

A graph $G = (V, E)$ is a split graph if the vertex set $V$ can be partitioned as $I \cup C$, where $I$ induces an independent set, and $C$ induces a clique in $G$. In what follows, we denote $G = (I \cup C, E)$, where $C$ is inclusion wise maximal, i.e., no vertices of $I$ is adjacent to all vertices of $C$.  Foldes and Hammer introduced the class of split graphs and characterized them as $(2K_2, C_4, C_5)$-free graphs in \cite{Foldes_1977}. Further, in \cite{Split_circle_graphs}, the class of split graphs restricted to permutation graphs and circle graphs are characterized in terms of forbidden induced subgraphs as per the following results.

\begin{theorem}[\cite{Split_circle_graphs}]\label{split_char}
	Let $G$ be a split graph. Then, we have the following characterizations: 
	\begin{enumerate}[label=\rm (\roman*)]
		\item \label{Split_permu_graph} $G$ is a permutation graph if and only if $G$ is a $\mathcal{B}$-free graph, where $\mathcal{B}$ is the class of graphs given in Fig. \ref{fig7}. 
		\item \label{forb_sub_split_cir} $G$ is a circle graph if and only if $G$ is a $\mathcal{C}$-free graph, where $\mathcal{C}$ is the class of graphs given in Fig. \ref{fig9} (the graph depictions are corrected using the PhD thesis of N. Pardal \cite{pardal2020}). Here, for convenience, the edges of the clique over  vertices $c, c_1, c_2, \ldots$ are not drawn completely.
	\end{enumerate}
\end{theorem}

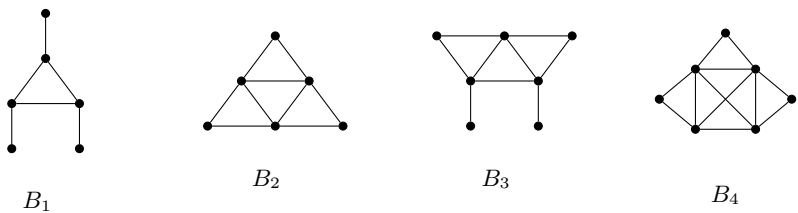
\begin{figure}[t]
	\centering
	\begin{minipage}{.25\textwidth}
		\centering
		\[\begin{tikzpicture}[scale=0.6]

			\vertex (1) at (0,0) [fill=black,label=left:$ $] {};
			\vertex (2) at (1.5,0) [fill=black,label=left:$ $] {};
			\vertex (3) at (0,1) [fill=black,label=left:$ $] {};	
			\vertex (4) at (1.5,1) [fill=black,label=left:$ $] {};
			\vertex (5) at (0.75,2) [fill=black,label=left:$ $] {};	
			\vertex (6) at (0.75,3) [fill=black,label=left:$ $] {};

			\path 
			(1) edge (3)
			(2) edge (4)
			(3) edge (5)
			(4) edge (5)
			(5) edge (6)
			(3) edge (4);

		\end{tikzpicture}\] 
		$B_1$	
	\end{minipage}%
	\begin{minipage}{.25\textwidth}
		\centering
		
		\[\begin{tikzpicture}[scale=0.6]

			\vertex (1) at (0,0) [fill=black,label=left:$ $] {};
			\vertex (2) at (1.5,0) [fill=black,label=left:$ $] {};
			\vertex (3) at (3,0) [fill=black,label=left:$ $] {};	
			\vertex (4) at (0.75,1) [fill=black,label=left:$ $] {};
			\vertex (5) at (2.25,1) [fill=black,label=left:$ $] {};	
			\vertex (6) at (1.5,2) [fill=black,label=left:$ $] {};

			\path
			(1) edge (2)
			(2) edge (3)
			(1) edge (4)
			(2) edge (4)
			(2) edge (5)
			(3) edge (5)
			(4) edge (5)
			(4) edge (6)
			(5) edge (6); 
			
		\end{tikzpicture}\]
		$B_2$
	\end{minipage}%
	\begin{minipage}{.25\textwidth}
		\centering
		
		\[\begin{tikzpicture}[scale=0.6]
			
			\vertex (1) at (0,0) [fill=black,label=left:$ $] {};
			\vertex (2) at (1.5,0) [fill=black,label=left:$ $] {};
			\vertex (3) at (0,1) [fill=black,label=left:$ $] {};	
			\vertex (4) at (1.5,1) [fill=black,label=left:$ $] {};
			\vertex (5) at (0.75,2) [fill=black,label=left:$ $] {};	
			\vertex (6) at (-0.75,2) [fill=black,label=left:$ $] {};	
			\vertex (7) at (2.25,2) [fill=black,label=left:$ $] {};

			\path 
			(1) edge (3)
			(2) edge (4)
			(3) edge (5)
			(4) edge (5)
			(5) edge (6)
			(3) edge (4)
			(3) edge (6)
			(5) edge (7)
			(4) edge (7);

		\end{tikzpicture}\]
		$B_3$
	\end{minipage}%
	\begin{minipage}{.25\textwidth}
		\centering
		
		\[\begin{tikzpicture}[scale=0.4]
			
			\vertex (1) at (0,0) [fill=black,label=below:$ $] {};
			\vertex (2) at (2,0) [fill=black,label=below:$ $] {};
			\vertex (3) at (2,2) [fill=black,label=right:$ $] {};
			\vertex (4) at (0,2) [fill=black,label=left:$ $] {};
			\vertex (5) at (1,3.2) [fill=black,label=above:$ $] {};
			\vertex (6) at (3.2,1) [fill=black,label=right:$ $] {};
			\vertex (7) at (-1.2,1) [fill=black,label=left:$ $] {};
			
			\path 
			(1) edge (2)
			(1) edge (7)
			(1) edge (4)
			(1) edge (3)
			(2) edge (3)
			(2) edge (4)
			(2) edge (6)
			(3) edge (4)
			(3) edge (5)
			(3) edge (6)
			(4) edge (5)
			(4) edge (7);
			
		\end{tikzpicture}\]
		$B_4$
	\end{minipage}%
	\caption{The family of graphs $\mathcal{B}$}
	\label{fig7}	
\end{figure}

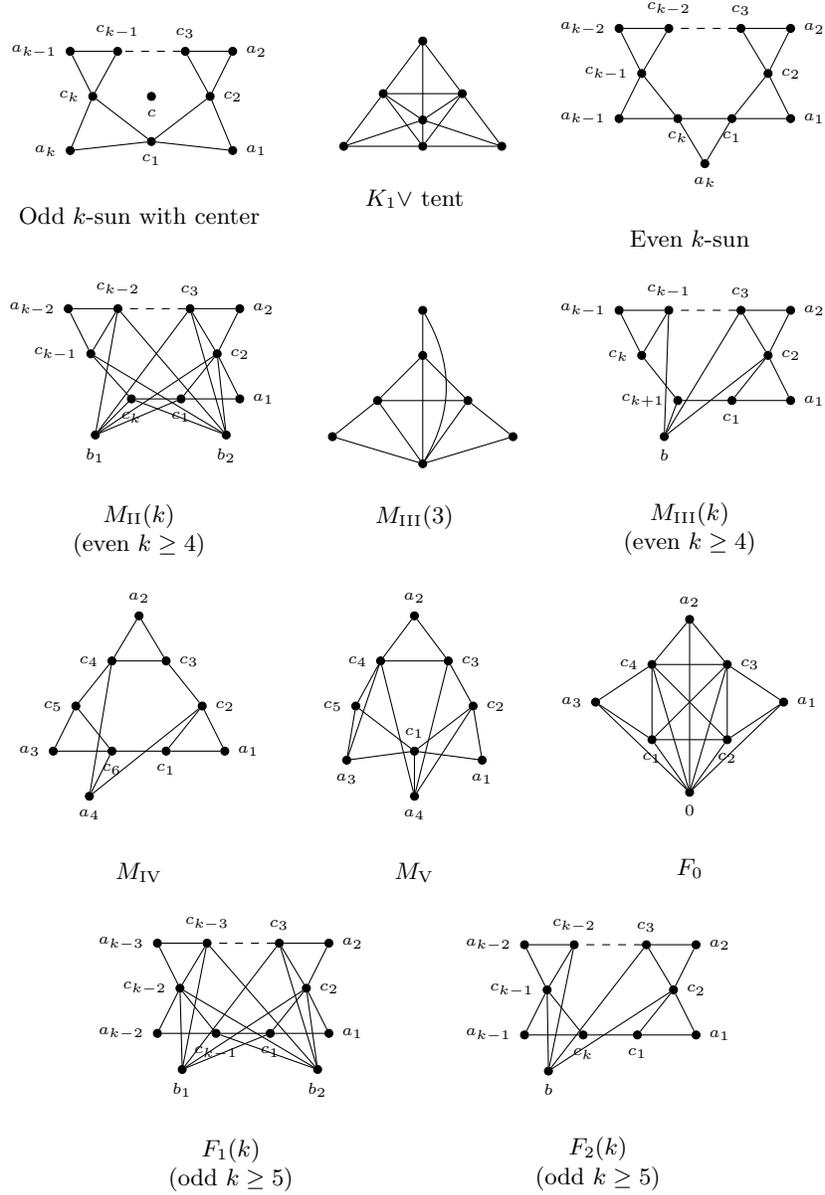
\begin{figure}[h!]
	\centering
	\begin{minipage}{.3\textwidth}
		\centering
		
		\[\begin{tikzpicture}[scale=0.6]

			\vertex (1) at (0,0) [fill=black,label=below:${\scriptstyle{c_1}}$] {};
			\vertex (2) at (1.3,1) [fill=black,label=right:${\scriptstyle{c_2}}$] {};
			\vertex (3) at (0.75,2) [fill=black,label=above:${\scriptstyle{c_3}}$] {};	
			\vertex (4) at (-0.75,2) [fill=black,label=above:${\scriptstyle{c_{k-1}}}$] {};
			\vertex (5) at (-1.3,1) [fill=black,label=left:${\scriptstyle{c_k}}$] {};	
			%\node (a) at (0,1.7) [label=above:${\cdots}$] {};
			\vertex (6) at (1.8,2) [fill=black,label=right:$\scriptstyle{a_2}$] {};
			\vertex (7) at (-1.8,2) [fill=black,label=left:$\scriptstyle{a_{k-1}}$] {};	
			\vertex (8) at (-1.8,-0.2) [fill=black,label=left:$ \scriptstyle{a_k}$] {};
			\vertex (9) at (1.8,-0.2) [fill=black,label=right:$\scriptstyle{a_1}$] {};
			\vertex (10) at (0,1) [fill=black,label=below:$\scriptstyle{c}$] {};	
			
			\path
			(1) edge (2)
			%			(1) edge (3)
			(2) edge (3)
			(1) edge (9)
			(2) edge (9)
			(2) edge (6)
			(3) edge (6)

			%			(1) edge (4)
			(1) edge (5)
			
			%			(2) edge (4)
			%			(2) edge (5)
			(3) edge[dashed] (4)
			%			(3) edge (5)
			(4) edge (5)
			
			(4) edge (7)
			(5) edge (7)
			(1) edge (8)
			(5) edge (8);
			
			%			(10) edge (1)
			%			(10) edge (2)
			%			(10) edge (3)
			%			(10) edge (4)
			%			(10) edge (5); 
			
		\end{tikzpicture}\]
		Odd $k$-sun	with center
	\end{minipage}%
	\begin{minipage}{.3\textwidth}
		\centering
		\[\begin{tikzpicture}[scale=0.7]

			\vertex (1) at (0,0) [fill=black,label=left:$ $] {};
			\vertex (2) at (1.5,0) [fill=black,label=left:$ $] {};
			\vertex (3) at (3,0) [fill=black,label=left:$ $] {};	
			\vertex (4) at (0.75,1) [fill=black,label=left:$ $] {};
			\vertex (5) at (2.25,1) [fill=black,label=left:$ $] {};	
			\vertex (6) at (1.5,2) [fill=black,label=left:$ $] {};	
			\vertex (7) at (1.5,0.5) [fill=black,label=left:$ $] {};
			
			\path
			(1) edge (2)
			(2) edge (3)
			(1) edge (4)
			(2) edge (4)
			(2) edge (5)
			(3) edge (5)
			(4) edge (5)
			(4) edge (6)
			(5) edge (6)
			(7) edge (1)
			(7) edge (2)
			(7) edge (3)
			(7) edge (4)
			(7) edge (5)
			(7) edge (6);	
		\end{tikzpicture}\] 
		$K_1 \vee$ tent  
	\end{minipage}%	
	\begin{minipage}{.3\textwidth}
		\centering
		
		\[\begin{tikzpicture}[scale=0.6]
			
			\vertex (1) at (0,0) [fill=black,label=below:${\scriptstyle{c_k}}$] {};
			\vertex (2) at (1.2,0) [fill=black,label=below:${\scriptstyle{c_1}}$] {};
			\vertex (3) at (2,1) [fill=black,label=right:${\scriptstyle{c_2}}$] {};	
			\vertex (4) at (1.4,2) [fill=black,label=above:${\scriptstyle{c_3}}$] {};
			\vertex (5) at (-0.2,2) [fill=black,label=above:${\scriptstyle{c_{k-2}}}$] {};	
			\vertex (6) at (-0.8,1) [fill=black,label=left:${\scriptstyle{c_{k-1}}}$] {};
			%\node (a) at (0.5,1.6) [label=above:${\cdots}$] {};	
			\vertex (7) at (2.5,2) [fill=black,label=right:${\scriptstyle{a_2}}$] {};
			\vertex (8) at (2.5,0) [fill=black,label=right:${\scriptstyle{a_1}}$] {};
			\vertex (9) at (-1.3,0) [fill=black,label=left:${\scriptstyle{a_{k-1}}}$] {};
			\vertex (10) at (-1.3,2) [fill=black,label=left:${\scriptstyle{a_{k-2}}}$] {};
			\vertex (11) at (0.6,-1) [fill=black,label=below:${\scriptstyle{a_{k}}}$] {};
			\path 
			(1) edge (2)
%			(1) edge (3)
%			(1) edge (4)
%			(1) edge (5)
			(1) edge (6)
			(2) edge (3)
%			(2) edge (4)
%			(2) edge (5)
%			(2) edge (6)
			(3) edge (4)
%			(3) edge (5)
%			(3) edge (6)
			(4) edge[dashed] (5)
%			(4) edge (6)
			(5) edge (6)
			(3) edge (7)
			(4) edge (7)
			(2) edge (8)
			(3) edge (8)
			(1) edge (9)
			(6) edge (9)
			(10) edge (5)
			(10) edge (6)
			(11) edge (1)
			(11) edge (2);

		\end{tikzpicture}\]
		Even $k$-sun
	\end{minipage}%

	\begin{minipage}{.3\textwidth}
		\centering
		
		\[\begin{tikzpicture}[scale=0.6]
			\vertex (1) at (0.1,0) [fill=black,label=below:${\scriptstyle{c_k}}$] {};
			\vertex (2) at (1.2,0) [fill=black,label=below:${\scriptstyle{c_1}}$] {};
			\vertex (3) at (2,1) [fill=black,label=right:${\scriptstyle{c_2}}$] {};	
			\vertex (4) at (1.4,2) [fill=black,label=above:${\scriptstyle{c_3}}$] {};
			\vertex (5) at (-0.2,2) [fill=black,label=above:${\scriptstyle{c_{k-2}}}$] {};	
			\vertex (6) at (-0.8,1) [fill=black,label=left:${\scriptstyle{c_{k-1}}}$] {};
			%\node (a) at (0.5,1.6) [label=above:${\cdots}$] {};	
			\vertex (7) at (2.5,2) [fill=black,label=right:$\scriptstyle{a_2}$] {};
			\vertex (8) at (2.5,0) [fill=black,label=right:$\scriptstyle{a_1}$] {};
			%\vertex (9) at (-1.3,0) [fill=black,label=left:${\scriptstyle{1'}}$] {};
			\vertex (10) at (-1.3,2) [fill=black,label=left:${\scriptstyle{a_{k-2}}}$] {};
			\vertex (11) at (-0.7,-0.8) [fill=black,label=below:${\scriptstyle{b_1}}$] {};
			\vertex (12) at (2.2,-0.8) [fill=black,label=below:${\scriptstyle{b_2}}$] {};
			\path 
			(1) edge (2)
			%			(1) edge (3)
			%			(1) edge (4)
			%			(1) edge (5)
			(1) edge (6)
			(2) edge (3)
			%			(2) edge (4)
			%			(2) edge (5)
			%			(2) edge (6)
			(3) edge (4)
			%			(3) edge (5)
			%			(3) edge (6)
			(4) edge[dashed] (5)
			%			(4) edge (6)
			(5) edge (6)
			(3) edge (7)
			(4) edge (7)
			(2) edge (8)
			(3) edge (8)
			%			(1) edge (9)
			%			(6) edge (9)
			(10) edge (5)
			(10) edge (6)
			(11) edge (1)
			(11) edge (2)
			(11) edge (3)
			(11) edge (4)
			(11) edge (5)
			(12) edge (1)
			(12) edge (3)
			(12) edge (4)
			(12) edge (5)
			(12) edge (6);
			
		\end{tikzpicture}\]
		$M_{\rm II}(k)$ \\
		(even $k \ge 4$)
	\end{minipage}%	
\begin{minipage}{.3\textwidth}
	\centering
	\[\begin{tikzpicture}[scale=0.6]
		
		\vertex (1) at (0,-0.4) [fill=black,label=left:$ $] {};
		\vertex (2) at (-1,1) [fill=black,label=left:$ $] {};
		\vertex (3) at (1,1) [fill=black,label=left:$ $] {};	
		\vertex (4) at (0,2) [fill=black,label=left:$ $] {};
		\vertex (5) at (0,3) [fill=black,label=left:$ $] {};	
		\vertex (6) at (-2,0.2) [fill=black,label=left:$ $] {};	
		\vertex (7) at (2,0.2) [fill=black,label=left:$ $] {};
		
		\path
		(1) edge (2)
		(1) edge (3)
		(1) edge (4)
		(5) edge[bend left] (1)
		(6) edge (1)
		(7) edge (1)
		(2) edge (3)	
		(2) edge (4)
		(3) edge (4)
		(4) edge (5)
		(6) edge (2)
		(7) edge (3);
		
	\end{tikzpicture}\] 
	$M_{\rm III}(3)$
\end{minipage}%	
	\begin{minipage}{.3\textwidth}
		\centering
		
		\[\begin{tikzpicture}[scale=0.6]
			
			\vertex (1) at (0,0) [fill=black,label= left:${\scriptstyle{c_{k+1}}}$] {};
			\vertex (2) at (1.2,0) [fill=black,label=below:${\scriptstyle{c_1}}$] {};
			\vertex (3) at (2,1) [fill=black,label=right:${\scriptstyle{c_2}}$] {};	
			\vertex (4) at (1.4,2) [fill=black,label=above:${\scriptstyle{c_3}}$] {};
			\vertex (5) at (-0.2,2) [fill=black,label=above:${\scriptstyle{c_{k-1}}}$] {};	
			\vertex (6) at (-0.8,1) [fill=black,label=left:${\scriptstyle{c_{k}}}$] {};
			%\node (a) at (0.5,1.6) [label=above:${\cdots}$] {};	
			\vertex (7) at (2.5,2) [fill=black,label=right:$\scriptstyle{a_2}$] {};
			\vertex (8) at (2.5,0) [fill=black,label=right:$\scriptstyle{a_1}$] {};
		%	\vertex (9) at (-1.3,0) [fill=black,label=left:${\scriptstyle{1}}$] {};
			\vertex (10) at (-1.3,2) [fill=black,label=left:${\scriptstyle{a_{k-1}}}$] {};
			\vertex (11) at (-0.3,-0.8) [fill=black,label=below:${\scriptstyle{b}}$] {};
			
			\path 
			(1) edge (2)
%			(1) edge (3)
%			(1) edge (4)
%			(1) edge (5)
			(1) edge (6)
			(2) edge (3)
%			(2) edge (4)
%			(2) edge (5)
%			(2) edge (6)
			(3) edge (4)
%			(3) edge (5)
%			(3) edge (6)
			(4) edge[dashed] (5)
%			(4) edge (6)
			(5) edge (6)
			(3) edge (7)
			(4) edge (7)
			(2) edge (8)
			(3) edge (8)
			(10) edge (5)
			(10) edge (6)
			(11) edge (1)
			(11) edge (3)
			(11) edge (4)
			(11) edge (5);

		\end{tikzpicture}\]
		$M_{\rm III}(k)$ \\
		(even $k \ge 4$)
	\end{minipage}%	
		
	\begin{minipage}{.3\textwidth}
		\centering
		\[\begin{tikzpicture}[scale=0.6]

			\vertex (1) at (0,0) [fill=black,label=below:$ \scriptstyle{c_6}$] {};
			\vertex (2) at (1.2,0) [fill=black,label=below:$\scriptstyle{c_1} $] {};
			\vertex (3) at (2,1) [fill=black,label=right:$\scriptstyle{c_2} $] {};	
			\vertex (4) at (1.2,2) [fill=black,label=right:$ \scriptstyle{c_3}$] {};
			\vertex (5) at (0,2) [fill=black,label=left:$\scriptstyle{c_4} $] {};	
			\vertex (6) at (-0.8,1) [fill=black,label=left:$ \scriptstyle{c_5}$] {};
			\vertex (8) at (2.5,0) [fill=black,label=right:${\scriptstyle{a_1}}$] {};
			\vertex (9) at (-1.3,0) [fill=black,label=left:${\scriptstyle{a_3}}$] {};
			\vertex (10) at (0.6,3) [fill=black,label=above:${\scriptstyle{a_2}}$] {};
			\vertex (11) at (-0.5,-1) [fill=black,label=below:${\scriptstyle{a_4}}$] {};
			
			\path
			(1) edge (2)
%			(1) edge (3)
%			(1) edge (4)
%			(1) edge (5)
			(1) edge (6)
			(2) edge (3)
%			(2) edge (4)
%			(2) edge (5)
%			(2) edge (6)
			(3) edge (4)
			(4) edge (5)
%			(4) edge (6)
			(11) edge (1)
			(11) edge (3)
			(11) edge (5)
			(10) edge (4)
			(1) edge (9)
			(2) edge (8)
			
%			(3) edge (5)
%			(3) edge (6)
			
			(5) edge (6)
			
			(3) edge (8)
			
			(6) edge (9)
			
			(10) edge (5)
			;		
		\end{tikzpicture}\] 
		$M_{\rm IV}$
	\end{minipage}%
	\begin{minipage}{.3\textwidth}
		\centering
		\[\begin{tikzpicture}[scale=0.6]

			\vertex (1) at (0,0) [fill=black,label=above:$\scriptstyle{c_1}$] {};
			\vertex (2) at (1.3,1) [fill=black,label=right:$\scriptstyle{c_2} $] {};
			\vertex (3) at (0.75,2) [fill=black,label=right:$\scriptstyle{c_3}$] {};	
			\vertex (4) at (-0.75,2) [fill=black,label=left:$\scriptstyle{c_4}$] {};
			\vertex (5) at (-1.3,1) [fill=black,label=left:$\scriptstyle{c_5}$] {};		
			\vertex (8) at (-1.5,-0.2) [fill=black,label=below:$ \scriptstyle{a_3}$] {};
			\vertex (9) at (1.5,-0.2) [fill=black,label=below:$ \scriptstyle{a_1}$] {};
			\vertex (10) at (0,-1) [fill=black,label=below:$\scriptstyle{a_4} $] {};	
			\vertex (11) at (0,3) [fill=black,label=above:$ \scriptstyle{a_2}$] {};
			
			\path
			(10) edge (1)
			(10) edge (2)
			(10) edge (3)
			(10) edge (4)
			(1) edge (2)
			(1) edge (5)
			(3) edge (4)
			(4) edge (5)
			(11) edge (4)
			(4) edge (8)
			(1) edge (8)
			(1) edge (9)
			(2) edge (3)
			(2) edge (9)
			(5) edge (8)
			(11) edge (3)
			;		
		\end{tikzpicture}\] 
		$M_{\rm V}$
	\end{minipage}%	
	\begin{minipage}{.3\textwidth}
		\centering
		\[\begin{tikzpicture}[scale=0.5]

			\vertex (1) at (0,0) [fill=black,label=below:${\scriptstyle{c_1}}$] {};
			\vertex (2) at (2,0) [fill=black,label=below:${\scriptstyle{c_2}}$] {};
			\vertex (3) at (2,2) [fill=black,label=right:${\scriptstyle{c_3}}$] {};
			\vertex (4) at (0,2) [fill=black,label=left:${\scriptstyle{c_4}}$] {};
			\vertex (5) at (1,3.2) [fill=black,label=above:${\scriptstyle{a_2}}$] {};
			\vertex (6) at (3.5,1) [fill=black,label=right:${\scriptstyle{a_1}}$] {};
			\vertex (7) at (-1.5,1) [fill=black,label=left:${\scriptstyle{a_3}}$] {};
			\vertex (8) at (1,-1.4) [fill=black,label=below:${\scriptstyle{0}}$] {};
			
			\path 
			(1) edge (2)
			(1) edge (7)
			(1) edge (4)
			(1) edge (3)
			(2) edge (3)
			(2) edge (4)
			(2) edge (6)
			(3) edge (4)
			(3) edge (5)
			(3) edge (6)
			(4) edge (5)
			(4) edge (7)
			(8) edge (1)
			(8) edge (2)
			(8) edge (3)
			(8) edge (4)
			(8) edge (5)
			(8) edge (6)
			(8) edge (7);	
		\end{tikzpicture}\] 
		$F_0$
	\end{minipage}%
	
	\begin{minipage}{.4\textwidth}
		\centering
		\[\begin{tikzpicture}[scale=0.6]
			
			\vertex (1) at (0,0) [fill=black,label=below:${\scriptstyle{c_{k-1}}}$] {};
			\vertex (2) at (1.2,0) [fill=black,label=below:${\scriptstyle{c_1}}$] {};
			\vertex (3) at (2,1) [fill=black,label=right:${\scriptstyle{c_2}}$] {};	
			\vertex (4) at (1.4,2) [fill=black,label=above:${\scriptstyle{c_3}}$] {};
			\vertex (5) at (-0.2,2) [fill=black,label=above:${\scriptstyle{c_{k-3}}}$] {};	
			\vertex (6) at (-0.8,1) [fill=black,label=left:${\scriptstyle{c_{k-2}}}$] {};
		%	\node (a) at (0.5,1.6) [label=above:${\cdots}$] {};	
			\vertex (7) at (2.5,2) [fill=black,label=right:${\scriptstyle{a_2}}$] {};
			\vertex (8) at (2.5,0) [fill=black,label=right:${\scriptstyle{a_1}}$] {};
			\vertex (9) at (-1.3,0) [fill=black,label=left:${\scriptstyle{a_{k-2}}}$] {};
			\vertex (10) at (-1.3,2) [fill=black,label=left:${\scriptstyle{a_{k-3}}}$] {};
			\vertex (11) at (-0.75,-0.8) [fill=black,label=below:${\scriptstyle{b_1}}$] {};
			\vertex (12) at (2.25,-0.8) [fill=black,label=below:${\scriptstyle{b_2}}$] {};
			\path 
			(1) edge (2)
%			(1) edge (3)
%			(1) edge (4)
%			(1) edge (5)
			(1) edge (6)
			(2) edge (3)
%			(2) edge (4)
%			(2) edge (5)
%			(2) edge (6)
			(3) edge (4)
%			(3) edge (5)
%			(3) edge (6)
			(4) edge[dashed] (5)
%			(4) edge (6)
			(5) edge (6)
			(3) edge (7)
			(4) edge (7)
			(2) edge (8)
			(3) edge (8)
			(1) edge (9)
			(6) edge (9)
			(10) edge (5)
			(10) edge (6)
			
			(11) edge (2)
			(11) edge (3)
			(11) edge (4)
			(11) edge (5)
			(11) edge (6)
			(12) edge (1)
			(12) edge (3)
			(12) edge (4)
			(12) edge (5)
			(12) edge (6);
		\end{tikzpicture}\] 
		$F_1(k)$ \\ (odd $k \ge 5$)
	\end{minipage}%
	\begin{minipage}{.4\textwidth}
		\centering
		\[\begin{tikzpicture}[scale=0.6]

			\vertex (1) at (0,0) [fill=black,label=below:${\scriptstyle{c_k}}$] {};
			\vertex (2) at (1.2,0) [fill=black,label=below:${\scriptstyle{c_1}}$] {};
			\vertex (3) at (2,1) [fill=black,label=right:${\scriptstyle{c_2}}$] {};	
			\vertex (4) at (1.4,2) [fill=black,label=above:${\scriptstyle{c_3}}$] {};
			\vertex (5) at (-0.2,2) [fill=black,label=above:${\scriptstyle{c_{k-2}}}$] {};	
			\vertex (6) at (-0.8,1) [fill=black,label=left:${\scriptstyle{c_{k-1}}}$] {};
			%\node (a) at (0.6,1.6) [label=above:${\cdots}$] {};	
			\vertex (7) at (2.5,2) [fill=black,label=right:${\scriptstyle{a_2}}$] {};
			\vertex (8) at (2.5,0) [fill=black,label=right:${\scriptstyle{a_1}}$] {};
			\vertex (9) at (-1.3,0) [fill=black,label=left:${\scriptstyle{a_{k-1}}}$] {};
			\vertex (10) at (-1.3,2) [fill=black,label=left:${\scriptstyle{a_{k-2}}}$] {};
			\vertex (11) at (-0.77,-0.8) [fill=black,label=below:${\scriptstyle{b}}$] {};
			
			\path 
			(1) edge (2)
%			(1) edge (3)
%			(1) edge (4)
%			(1) edge (5)
			(1) edge (6)
			(2) edge (3)
%			(2) edge (4)
%			(2) edge (5)
%			(2) edge (6)
			(3) edge (4)
%			(3) edge (5)
%			(3) edge (6)
			(4) edge[dashed] (5)
%			(4) edge (6)
			(5) edge (6)
			(3) edge (7)
			(4) edge (7)
			(2) edge (8)
			(3) edge (8)
			(1) edge (9)
			(6) edge (9)
			(10) edge (5)
			(10) edge (6)

			(11) edge (3)
			(11) edge (4)
			(11) edge (5)
			(11) edge (6);
		\end{tikzpicture}\] 
		$F_2(k)$ \\ (odd $k \ge 5$)
	\end{minipage}%

	\caption{The family $\mathcal{C}$ of split graphs}
	\label{fig9}	
\end{figure}

Split graphs which are transitively orientable are called split comparability graphs. Recall from \cite{split_orders2004} that a partial order is called a split order if the corresponding comparability graph is a split graph. It was proved in \cite{split_orders2004} that the dimension of a split order is at most three. Hence, in view of \cite[Corollary 2]{khyodeno2}, it is evident that the \textit{prn} of a split comparability graph is at most three. The class of split comparability graphs are characterized in terms of the forbidden induced subgraphs as per the following result.

\begin{theorem}[\cite{Golumbicbook_2004}] \label{Split_comp_graph}
	Let $G$ be a split graph. Then, $G$ is a comparability graph if and only if $G$ contains no induced subgraph isomorphic to $B_1$, $B_2$, or $B_3$ (depicted in Fig. \ref{fig7}).
\end{theorem}

In \cite{Kitaev_2021}, it was shown that not all split graphs are word-representable and a characterization for the class of word-representable split graphs was found. While the recognition problem of word-representable graphs is NP-complete, in \cite{Kitaev_2024}, it was proved that recognizing the word-representable split graphs can be done in polynomial time.  The following characterization of word-representable split graphs is a key tool for our work. In what follows, for any two integers $a \le b$, we denote the set of integers $\{a, a+1, \ldots, b\}$ by $[a, b]$.

\begin{theorem}[\cite{Kitaev_2021,Kitaev_2024}]\label{Word_split_graph}
	Let $G = (I \cup C, E)$ be a split graph. Then, $G$ is word-representable if and only if the vertices of $C$ can be labeled from $1$ to $k = |C|$ in such a way that for each $a, b \in I$ the following holds.
	\begin{enumerate}[label=\rm (\roman*)]
		\item \label{point_1} Either $N(a) = [1, m] \cup [n, k]$, for $m < n$, or $N(a) = [l, r]$, for $l \le r$.
		\item \label{point_2} If $N(a) = [1, m] \cup [n, k]$ and $N(b) = [l, r]$, for $m < n$ and $l \le r$, then $l > m$ or $r < n$.
		\item \label{point_3} If $N(a) = [1, m] \cup [n, k]$ and $N(b) = [1, m'] \cup [n', k]$, for $m < n$ and $m' < n'$, then $m' < n$ and $m < n'$.
	\end{enumerate}
\end{theorem}

Further, in \cite{tithi_splitcom}, Theorem \ref{Word_split_graph} is extended to split comparability graphs as stated in the following theorem.

\begin{theorem}[\cite{tithi_splitcom}]\label{coro_3}
	Let $G = (I \cup C, E)$ be a split graph. Then, $G$ is transitively orientable if and only if the vertices of $C$ can be labeled from $1$ to $k = |C|$ such that the following properties $({\rm i})-({\rm v})$ hold: For $a, b \in I$,
	\begin{enumerate}[label=\rm (\roman*)]
		\item \label{pt_1} The neighborhood $N(a)$ has one of the following forms: $[1, m] \cup [n, k]$ for $m < n$, $[1, r]$ for $r < k$, or $[l, k]$ for $l > 1$.
		
		\item \label{pt_2} If $N(a) = [1, r]$ and $N(b) = [l, k]$, for $r < k$ and $l > 1$, then $r < l$.
		
		\item \label{pt_3} If $N(a) = [1, m] \cup [n, k]$ and $N(b) = [1, r]$,  for $r < k$ and $m < n$, then $r < n$.
		
		\item \label{pt_4} If $N(a) = [1, m] \cup [n, k]$ and $N(b) = [l, k]$, for $l > 1$ and $m < n$, then $m < l$.
		
		\item \label{pt_5} If $N(a) = [1, m] \cup [n, k]$ and  $N(b) = [1, m'] \cup [n', k]$, for $m < n$ and $m' < n'$, then $m < n'$ and $m' < n$.
	\end{enumerate}
\end{theorem}

\section{Representation Number}\label{r_no}

Let $G = (I \cup C, E)$ be a word-representable split graph. Then, in view of Theorem \ref{Word_split_graph}, the vertices of $C$ can be labeled from $1$ to $k = |C|$ such that it satisfies the three properties. Moreover, such labelling of the vertices of $C$ can be found in polynomial time \cite{Kitaev_2024}. In what follows, we consider the aforementioned labelling of the vertices of $C$. Note that $N(a) \subseteq C$, for each $a \in I$. We now consider the following sets:
\begin{align*}
	A & = \{a \in I \mid N(a) = [1, m] \cup [n, k], \ \text{for some} \ m < n\} \\
	B & = \{a \in I \mid N(a) = [l, r], \ \text{for some} \ l \le r\}
\end{align*} 
Then, from Theorem \ref{Word_split_graph}, we have $A \cap B = \varnothing$ and $I = A \cup B$. For $a \in A$, let $N(a) = [1, m_a] \cup [n_a, k]$ and for $a \in B$, let $N(a) = [l_a, r_a]$. By Theorem \ref{Word_split_graph} \ref{point_3}, since $m_a < n_b$ for any $a, b \in A$, we have the following remark.
\begin{remark}\label{prop_1}
	$\max\{m_a \mid a \in A \} < \min\{n_a \mid a \in A\}$.
\end{remark}

\subsection{3-Uniform Word-Representant of $G$}

Given a split graph $G = (I \cup C, E)$ with semi-transitive orientation (and hence, a labelling of the vertices of $C$), Algorithm \ref{Algo_1} constructs a $3$-uniform word-representant of $G$. In what follows, $w$ refers to the output of Algorithm \ref{Algo_1}.

\begin{algorithm}[!htb]\label{Algo_1}
	\caption{Constructing a 3-uniform word-representant of a split graph.}
	\label{algo-1}
	\KwIn{A word-representable split graph $G = (I \cup C, E)$ with $1, \ldots, k$ as the labels of the vertices of $C$.}
	\KwOut{A 3-uniform word $w$ representing $G$.}

	Initialize $p_1$, $p_2$ and $p_3$ with the permutation $12 \cdots k$ and update them with the vertices of $I$, as per the following.\\
	Initialize $d = 1$.\\
	\For{$a \in I$}{\If{$a \in A$}{ Let $N(a) = [1, m] \cup [n, k]$.\\
	\If{$m > d$}{Assign $d = m$.}

	Replace $m$ in $p_1$ with $ma$.\\
	Replace $n$ in $p_2$ with $an$.}
	\Else{Let $N(a) = [l, r]$.\\
	Replace $l$ in $p_1$ with $al$.\\
	Replace $r$ in $p_2$ with $ra$.}
	}
	Replace $d$ in $p_3$ with $d (p_1|_{A})^R$.\\
	\Return{{\rm the 3-uniform word} $p_1(p_1|_{B})^Rp_2p_3$}	
\end{algorithm}

\begin{remark}\label{position_C}
	In Algorithm \ref{Algo_1}, note that the order of the vertices of $C$, i.e., $12 \cdots k$, in the words $p_1, p_2$, and $p_3$ do not alter when they are updated. Hence, we have  $w|_{C}  = 12 \cdots k12 \cdots k12 \cdots k$.
\end{remark}

\begin{remark}\label{Occurrence_AB}
	Note that both the words $p_1$ and $p_2$, eventually, are permutations on the set $V$, while the word $p_3$ is a permutation only on $C \cup A$. Observe that the elements of $B$ do not occur in $p_3$.
\end{remark}

\begin{lemma}\label{lemma_0}
	If $a \in A$, then $m_aan_a$ is a subword of each of $p_1$, $p_2$, and $p_3$.
\end{lemma}

\begin{proof}
	Note that  $N(a) = [1, m_a] \cup [n_a, k]$, for $m_a < n_a$. Thus, as per line 1 of Algorithm \ref{Algo_1}, we have $m_an_a \ll p_1$, $m_an_a \ll p_2$, and $m_an_a \ll p_3$. As $a \in A$, replacing $m_a$ in $p_1$ with $m_aa$ (see line 9), and replacing $n_a$ in $p_2$ with $an_a$ (see line 10), we have $m_aan_a \ll p_1$, and $m_aan_a \ll p_2$, respectively. Further, in view of Remark \ref{prop_1}, we have $m_a \le d < n_a$ so that we have $m_a d n_a \ll p_3$. Now, as per line 18 of Algorithm \ref{Algo_1}, we have $m_aan_a \ll m_a d an_a \ll p_3$. \qed
\end{proof}

\begin{lemma}\label{lemma_1}
	For $a, b \in V$, if $a$ and $b$ are adjacent in $G$, then $a$ and $b$ alternate in $w$.
\end{lemma}

\begin{proof}
	 As $I$ is an independent set in $G$, note that both $a$ and $b$ cannot be in $I$. 	 
	 Suppose both $a, b \in C$. As the vertices of $C$ are labeled from $1$ to $k$, we have $a = i$ and $b = j$, for some $1 \le i, j \le k$. Thus, in view of Remark \ref{position_C}, we have $w|_{\{a, b\}} = ababab$, if $a < b$, or $w|_{\{a, b\}} = bababa$, if $b < a$. Hence, in this case, $a$ alternates with $b$ in $w$. If $a \in I$ and $b \in C$, then we consider the following cases:
	\begin{itemize}
		\item Case 1: $a \in A$. Note that $b \in N(a) = [1, m_a] \cup [n_a, k]$, for $m_a < n_a$. 
		If $b \in [1, m_a]$, as per line 1 of Algorithm \ref{Algo_1}, we have $bm_a \ll p_1$, $bm_a \ll p_2$, and $bm_a \ll p_3$. Hence, in view of Lemma \ref{lemma_0}, we have $bm_aan_a \ll p_1$, $bm_aan_a \ll p_2$, and $bm_aan_a \ll p_3$. As $p_1p_2p_3 \ll w$ and $w$ is a $3$-uniform word, we have $w|_{\{a, b\}} = bababa$. Thus, $a$ and $b$ alternate in $w$.
		Similarly, if $b \in [n_a, k]$, we can observe that $m_aan_ab \ll p_1$ and $m_aan_ab \ll p_2$, and $m_aan_ab \ll p_3$. Hence, $w|_{\{a, b\}} = ababab$ so that $a$ and $b$ alternate in $w$.
		
    	\item Case 2: $a \in B$. Note that $b \in N(a) =[l_a, r_a]$, for $l_a \le r_a$. Thus, by line 1 of Algorithm \ref{Algo_1}, we have $l_abr_a \ll p_1$, $l_abr_a \ll p_2$, and $l_abr_a \ll p_3$. As $a \in B$, replacing $l_a$ in $p_1$ with $al_a$ (see line 14), and replacing $r_a$ in $p_2$ with $r_aa$ (see line 15), we have $al_abr_a \ll p_1$ and $l_abr_aa \ll p_2$, respectively. Finally, as $w = p_1(p_1|_{B})^Rp_2p_3$, we have $w|_{\{a, b\}} = ababab$. Thus, $a$ and $b$ alternate in $w$.	
	\end{itemize}
\qed 
\end{proof}

\begin{lemma}\label{lemma_2}
	If $a, b \in I$, then $a$ and $b$ do not alternate in $w$.
\end{lemma}

\begin{proof}
	  Since $I = A \cup B$, we consider the following cases.
	\begin{itemize}
		\item Case 1: $a, b \in A$. Since both $a$ and $b$ appear in $p_1$, without loss of generality, suppose $ab \ll p_1$. Since $(p_1|_A)^R \ll p_3$ ( by line 18 of Algorithm \ref{Algo_1}), we have $ba \ll p_3$. 
		Thus, if $ab \ll p_2$, we have $ababba \ll p_1p_2p_3 \ll w$, and if $ba \ll p_2$, we have $abbaba \ll p_1p_2p_3 \ll w$. As $w$ is a $3$-uniform word, in any case, we can see that $a$ and $b$ do not alternate in $w$.
		
		\item Case 2: $a, b \in B$. Since both $a$ and $b$ appear in $p_1$, without loss of generality, suppose $ab \ll p_1$. Then, we have $ba \ll (p_1|_{B})^R$. Since $w = p_1(p_1|_{B})^Rp_2p_3$, in view of Remark \ref{Occurrence_AB}, $abba$ is a factor of $w|_{\{a, b\}}$. Hence, $a$ and $b$ do not alternate in $w$.
		
		\item Case 3:  $a \in A$ and $b \in B$. Note that $N(a) = [1, m_a] \cup [n_a, k]$, for $m_a < n_a$, and $N(b) = [l_b, r_b]$, for $l_b \le r_b$. Then, from Theorem \ref{Word_split_graph} \ref{point_2}, we have $l_b > m_a$ or $r_b < n_a$. Accordingly, we consider the following cases.

		\item Subcase 3.1: $l_b > m_a$. From line 1 of Algorithm \ref{Algo_1}, we have $m_al_b \ll p_1$. As $a \in A$, replacing $m_a$ in $p_1$ with $m_aa$ we have, $m_aal_b \ll p_1$ (see line 9). Further, as $b \in B$, replacing $l_b$ in $p_1$ with $bl_b$ we have, $m_aabl_b \ll p_1$ (see line 14) so that $ab \ll p_1$. Since $w = p_1(p_1|_{B})^Rp_2p_3$, we have $abbp_2p_3 \ll w$. Thus, in view of Remark \ref{Occurrence_AB},  $abb$ is a factor of $w|_{\{a, b\}}$. Hence, $a$ and $b$ do not alternate in $w$.
		
		\item Subcase 3.2: $r_b < n_a$. From line 1 of Algorithm \ref{Algo_1}, we have $r_bn_a \ll p_2$. As $a \in A$, replacing $n_a$ in $p_2$ with $an_a$, we have $r_ban_a \ll p_2$ (see line 10). Further, as $b \in B$, replacing $r_b$ in $p_2$ with $r_bb$, we have $r_bban_a \ll p_2$ (see line 15) so that $ba \ll p_2$. Since $w = p_1(p_1|_{B})^Rp_2p_3$, in view of Remark \ref{Occurrence_AB}, $baa$ is a factor of $w|_{\{a, b\}}$. Hence, $a$ and $b$ do not alternate in $w$.
\end{itemize}
Thus, in any case, $a$ and $b$ do not alternate in $w$. \qed
\end{proof}

\begin{lemma}\label{lemma_3}
	For $a \in I$ and $b \in C$, if $a$ and $b$ are not adjacent in $G$, then $a$ and $b$ do not alternate in $w$.
\end{lemma}

\begin{proof}
	 As $I = A \cup B$, we consider the following cases.
	\begin{itemize}
		\item Case 1: $a \in A$. Note that $N(a) = [1, m_a] \cup [n_a, k]$, for $m_a < n_a$. As $b \in C$ and $b \notin N(a)$, we have $m_a < b < n_a$. Thus, $m_abn_a \ll p_1$, $m_abn_a \ll p_2$ and $m_abn_a \ll p_3$. Accordingly, replacing $m_a$ in $p_1$ with $m_aa$ (see line 9), and replacing $n_a$ in $p_2$ with $an_a$ (see line 10), we have $m_aabn_a \ll p_1$ and $m_aban_a \ll p_2$, respectively. Thus, $ab \ll p_1$ and $ba \ll p_2$ so that we have $abba \ll p_1p_2$.  Further, since $w = p_1(p_1|_{B})^Rp_2p_3$, in view of Remark \ref{Occurrence_AB},  it is evident that $abba$ is a factor of $w|_{\{a, b\}}$. Hence, $a$ and $b$ do not alternate in $w$.
		
		\item Case 2: $a \in B$. Note that $N(a) = [l_a, r_a]$, for $l_a \le r_a$. As $b \in C$ and $b \notin N(a)$, we have $b < l_a$ or $b > r_a$. Suppose $b < l_a$. Then, we have $bl_ar_a \ll p_1$, $bl_ar_a \ll p_2$, and $bl_ar_a \ll p_3$. Accordingly, replacing $l_a$ in $p_1$ with $al_a$ (by line 14), and replacing $r_a$ in $p_2$ with $r_aa$ (by line 15), we have $bal_ar_a \ll p_1$ and $bl_ar_aa \ll p_2$, respectively. Thus, $ba \ll p_1$ and $ba \ll p_2$. Further, since $w = p_1(p_1|_{B})^Rp_2p_3$, and $a \ll (p_1|_{B})^R$, $b \ll p_3$, we have $baabab \ll w$ so that $a$ and $b$ do not alternate in $w$. Similarly, if $b > r_a$, we can observe that $abaabb \ll w$ so that $a$ and $b$ do not alternate in $w$.	
	\end{itemize}
Thus, in any case, $a$ and $b$ do not alternate in $w$. \qed
\end{proof}

\begin{theorem}\label{3-word-rep}
	Let $G = (I \cup C, E)$ be a word-representable split graph. Then, $\mathcal{R}(G) \le 3$.
\end{theorem}

\begin{proof}
	Note that, if $a$ and $b$ are non-adjacent vertices of $G$, then either both $a, b \in I$ or one of $a$ and $b$ belongs to $I$ and the other one belongs to $C$. Thus, by lemmas \ref{lemma_1}, \ref{lemma_2}, and \ref{lemma_3}, we have the 3-uniform word $w$ represents $G$. Hence, $\mathcal{R}(G) \le 3$.  \qed   
\end{proof}

\subsection{Characterization of $\mathcal{R}(G) = 3$}

In the following, we characterize the word-representable split graphs with representation number three (cf. Theorem \ref{Rep_no_split}). Further, we give a characterization for split comparability graphs with representation number three (cf. Theorem \ref{char_split_com}).  

Consider the graph even $k$-sun from the class $\mathcal{C}$ depicted in Fig. \ref{fig9}. For an even integer $k \ge 4$, an even $k$-sun is the split graph $(I \cup C, E)$ with $I = \{a_1, \ldots, a_k\}$ and $C = \{c_1, \ldots, c_k\}$ such that $N(a_i) = \{c_i, c_{i+1}\}$, for each $1 \le i \le k-1$, and $N(a_k) = \{c_1, c_k\}$. We now recall the result \cite[Theorem 9]{Kitaev_2021} on word-representability of an even $k$-sun. 

\begin{lemma}[\cite{Kitaev_2021}]\label{evnksun}
	For $k \ge 4$, an even $k$-sun is a word-representable graph.
\end{lemma}

For an odd integer $k \ge 3$, an odd $k$-sun with center (see the depiction in Fig.~\ref{fig9}) is the split graph $(I \cup C, E)$ with $I = \{a_1, \ldots, a_k\}$ and $C = \{c, c_1, \ldots, c_k\}$ such that $N(a_i) = \{c_i, c_{i+1}\}$, for each $1 \le i \le k-1$ and $N(a_k) = \{c_1, c_k\}$. The vertex $c$ is called the center. We now recall the result \cite[Theorem 11]{Kitaev_2021} on word-representability of an odd $k$-sun with center. 

\begin{lemma}[\cite{Kitaev_2021}]
	For $k \ge 3$, an odd $k$-sun with center is a non-word-representable graph.
\end{lemma}

Note that the graph $M_{\rm V}$ from the class $\mathcal{C}$ (cf. Fig. \ref{fig9}) is isomorphic to the graph $T_8$ depicted in \cite[Fig. 3]{Chen_2022}. Accordingly, the result  \cite[Theorem 19]{Chen_2022} can be stated as per the following.

\begin{lemma}[\cite{Chen_2022}]
	 The graph $M_V$ is non-word-representable.
\end{lemma}

\begin{lemma}\label{WRF0}
	The graph $F_0$ is word-representable.
\end{lemma}

\begin{proof}
	The graph $F_0$ is obtained from the graph $B_4$, given in Fig. \ref{fig7}, by adding an all-adjacent vertex. In view of \cite[Lemma 3]{Hallsorsson_2011}, $F_0$ is a word-representable graph if and only if $B_4$ is a comparability graph. Note that $B_4$ is a split graph and it contains no induced subgraph isomorphic to $B_1, B_2$ or $B_3$ (refer Fig. \ref{fig7}). Thus, by Theorem \ref{Split_comp_graph}, $B_4$ is a comparability graph. Hence, the result follows.  \qed
\end{proof}

\begin{lemma}
	The graphs $M_{\rm III}(3)$ and  $K_1 \vee$ {\rm tent} are non-word-representable graphs.
\end{lemma}

\begin{proof}
    Note that the graphs $M_{\rm III}(3)$ and $K_1 \vee$ {\rm tent} are obtained from the graphs $B_1$ and $B_2$ (see Fig. \ref{fig7}), respectively, by adding an all-adjacent vertex. In a word-representable graph, the neighborhood of every vertex induces a comparability graph (see \cite[Theorem 9]{MR2467435} or \cite[Lemma 3]{Hallsorsson_2011}). However, since none of $B_1$ and $B_2$ is a comparability graph, $M_{\rm III}(3)$ and $K_1 \vee$ {\rm tent} are not word-representable graphs.  \qed	
\end{proof}

For an even integer $k \ge 4$, the graph $M_{\rm II}(k)$ (given in Fig. \ref{fig9}) is the split graph $(I \cup C, E)$ with $C = \{c_1, \ldots, c_k\}$ and $I = \{b_1, b_2, a_1, \ldots, a_{k-2}\}$ such that $N(b_1) = \{c_1, \ldots, c_{k-2}, c_k\}$, $N(b_2) = \{c_2, \ldots, c_k\}$ and, for $1 \le i \le k-2$, $N(a_i) = \{c_i, c_{i+1}\}$. 

\begin{lemma}
	For $k \ge 4$, the graph $M_{\rm II}(k)$ is non-word-representable.
\end{lemma} 

\begin{proof}
	On the contrary, suppose $M_{\rm II}(k)$ is a word-representable graph. Then, in view of Theorem \ref{Word_split_graph}, the vertices of $C$ can be labeled from $1$ to $k = |C|$ such that, for $a \in I$, either $a \in A$, i.e., $N(a) = [1, m] \cup [n, k]$ (with $m < n$) or $a \in B$, i.e., $N(a) = [l, r]$ (with $l \le r$). We derive a contradiction in each of the following cases, by considering the possibilities of $b_1, b_2$ belonging to the sets $A, B$. 	
	\begin{itemize}
		\item Case 1: $b_1, b_2 \in A$. Note that $\deg(b_1) = k-1 = \deg(b_2)$ and $N(b_1) \neq N(b_2)$. In view of Theorem \ref{Word_split_graph} \ref{point_3}, without loss of generality, for some $m \ge 1$ with $m +2 < k$, let $N(b_1) = [1, m] \cup [m+2, k]$  and $N(b_2) = [1, m+1] \cup [m+3, k]$. Accordingly, the vertices $c_1$ and $c_{k-1}$ are labeled as $m+2$ and $m+1$, respectively. Further, for each $1 \le i \le k-2$, since the degree of $a_i$ is two, we have $N(a_i) = [j, j+1]$, for some $1 \le j < k$, or $N(a_i) = [1, 1] \cup [k, k]$. Note that $N(a_1) = \{c_1, c_2\}$ and $c_1$ is labeled as $m+2$ $(< k)$. Since $m+1$ is already used as the label of $c_{k-1}$, the vertex $c_2$ should be labeled as $m+3$. Continuing with this argument, the vertices $c_3, \ldots, c_{k-m-1}$ should be labeled as $m + 4, \ldots, k$, respectively. On the other hand, since $N(a_{k-2}) = \{c_{k-2}, c_{k-1}\}$ and $m+2$ is used as the label of $c_{1}$,  the vertex $c_{k-2}$ should be labeled as $m$. Again continuing with this argument, the vertices $c_{k-3}, \ldots, c_{k-m}$ are labeled as $m-1, \ldots, 2$, respectively. This shows that $N(a_{k-m-1}) = \{c_{k-m-1}, c_{k-m}\} = \{k, 2\}$, a contradiction to Theorem \ref{Word_split_graph}\ref{point_1}, as $k \ge 4$.  
		
		\item Case 2: $b_1, b_2 \in B$. Since $\deg(b_1) = \deg(b_2) = k-1$, and $N(b_1) \neq N(b_2)$, one of $N(b_1)$ and $N(b_2)$ is $[2, k]$, and the other one is $[1, k-1]$. Without loss of generality, suppose that $N(b_1) = [2, k]$ and $N(b_2) = [1, k-1]$. Thus, $c_1$ and $c_{k-1}$ are labeled as $k$ and $1$, respectively. Further, for each $1 \le i \le k-2$, since the degree of $a_i$ is two, we have $N(a_i) = [j, j+1]$, for some $1 \le j < k$, or $N(a_i) = [1, 1] \cup [k, k]$. Since $N(a_{k-2}) = \{c_{k-2}, c_{k-1}\}$, the vertex $c_{k-2}$ should be labeled as $2$. Continuing the same argument, $3, \ldots, k-2$ must be the labels of the vertices $c_{k-3}, \ldots, c_2$, respectively. Thus, we have $N(a_1) = \{c_1, c_2\} = \{k, k-2\}$, a contradiction to Theorem \ref{Word_split_graph}\ref{point_1}, as $k \ge 4$.
		
		\item Case 3: $b_1 \in A$ and $b_2 \in B$. Since $\deg(b_2) = k-1$, we have $N(b_2) = [1, k-1]$ or $[2, k]$. Suppose that $N(b_2) = [1, k-1]$. Since $\deg(b_1) = k-1$, in view of Theorem \ref{Word_split_graph}\ref{point_2}, we have $N(b_1) = [1, k-2] \cup [k, k]$. Thus, $k$ and $k-1$ must be the labels of $c_1$ and $c_{k-1}$, respectively. Since $N(a_{k-2}) = \{c_{k-2}, c_{k-1}\}$, $c_{k-2}$ should be labeled as $k-2$. Note that as $N(a_i) = \{c_i, c_{i+1}\}$, for each $2 \le i \le k-3$, continuing the same argument, the vertices $c_{k-3}, \ldots, c_2$ must receive the labels $k-3, \ldots, 2$, respectively. Thus, $N(a_1) = \{c_1, c_2\} = \{k, 2\}$, a contradiction to Theorem \ref{Word_split_graph}\ref{point_1}. Similarly, if $N(b_2) = [2, k]$, we will arrive at a contradiction.
	\end{itemize}
Hence, the graph $M_{\rm II}(k)$ is non-word-representable. \qed

\end{proof}

For an even integer $k \ge 4$, the graph $M_{\rm III}(k)$ (depicted in Fig. \ref{fig9}) is the split graph $(I \cup C, E)$ with $C = \{c_1, \ldots, c_{k+1}\}$ and $I = \{b, a_1, \ldots, a_{k-1}\}$ such that $N(b) = \{c_2, \ldots, c_{k-1}, c_{k+1}\}$ and, for $1 \le i \le k-1$, $N(a_i) = \{c_i, c_{i+1}\}$.

\begin{lemma}
	For $k \ge 4$, the graph $M_{\rm III}(k)$ is non-word-representable.
\end{lemma} 

\begin{proof}	
	On the contrary, suppose $M_{\rm III}(k)$ is a word-representable graph. Then, in view of Theorem \ref{Word_split_graph}, the vertices of $C$ can be labeled from $1$ to $k+1 = |C|$ such that, for $a \in I$, either $a \in A$, i.e., $N(a) = [1, m] \cup [n, k+1]$ (with $m < n$) or $a \in B$, i.e., $N(a) = [l, r]$ (with $l \le r$). We derive a contradiction in each of the following cases, by considering the possibilities of $b$ belonging to the sets $A, B$.  	
	\begin{itemize}
		\item Case 1: $b \in A$. Since $\deg(b) = k-1$, we have $N(b) = [1, m] \cup [m+3, k+1]$. Thus, one of $c_1$ and $c_k$ is labeled as $m+1$, and the other one is labeled as $m+2$. In view of the symmetry, suppose $c_1$ and $c_k$ are labeled as $m+2$ and $m+1$, respectively. Since $N(a_1) = \{c_1, c_2\}$ and $m+1$ is already used as the label of $c_k$, the vertex $c_2$ should be labeled as $m+3$. Continuing the same argument, we conclude that the vertices $c_3, \ldots, c_{k-m}$ are labeled as $m+4, \ldots, k+1$, respectively. On the other hand, as $N(a_{k-1}) = \{c_{k-1}, c_k\}$, the vertex $c_{k-1}$ should be labeled as $m$. Accordingly, the vertices $c_{k-2}, \ldots, c_{k-m+1}$ are labeled as $m-1, \ldots, 2$, respectively. Thus, $N(a_{k-m}) = \{c_{k-m}, c_{k-m+1}\} = \{k+1, 2\}$, a contradiction to Theorem \ref{Word_split_graph}\ref{point_1}. 
		
		\item Case 2: $b \in B$. Since $\deg(b) = k-1$, we have $N(b) = [1, k-1]$ or $[2, k]$, or $[3, k+1]$. Suppose $N(b) = [1, k-1]$. Thus, one of $c_1$ and $c_k$ is labeled as $k$, and the other one is labeled as $k+1$. Without loss of generality, suppose $c_1$ and $c_k$ are labeled as $k$ and $k+1$, respectively. Since $N(a_1) = \{c_1, c_2\}$, the vertex $c_2$ should be labeled as $k-1$. Accordingly, the vertices $c_3, \ldots, c_{k-1}$ are labeled as $k-2, \ldots, 2$, respectively. Thus, $N(a_{k-1}) = \{c_{k-1}, c_k\} = \{2, k+1\}$, a contradiction to Theorem \ref{Word_split_graph}\ref{point_1}. Similarly, if $N(b) = [2, k]$ or $[3, k+1]$, we will arrive at a contradiction.
	\end{itemize}
Hence, there is no labelling of the vertices of $C$ which satisfies Theorem \ref{Word_split_graph} so that the graph $M_{\rm III}(k)$ is non-word-representable. \qed
\end{proof}

The graph $M_{\rm IV}$ is the split graph $(I \cup C, E)$ with $C = \{c_1, c_2, c_3, c_4, c_5, c_6\}$ and $I = \{a_1, a_2, a_3, a_4\}$ such that $N(a_1) = \{c_1, c_2\}, N(a_2) = \{c_3, c_4\}, N(a_3) = \{c_5, c_6\}$, and $N(a_4) = \{c_2, c_4, c_6\}$.

\begin{lemma}
	The graph $M_{\rm IV}$ is a non-word-representable graph.
\end{lemma} 

\begin{proof}
	 On the contrary, suppose $M_{\rm IV}$ is word-representable. Then, in view of Theorem \ref{Word_split_graph}, the vertices of $C$ can be labeled from $1$ to $6 = |C|$ such that, for $a \in I$, either $a \in A$, i.e., $N(a) = [1, m] \cup [n, 6]$ (with $m < n$) or $a \in B$, i.e., $N(a) = [l, r]$ (with $l \le r$).  	
	
	For $1 \le i \le 3$, since degree of $a_i$ is two, we have $N(a_i) = \{1, 6\}$, or $\{j, j+1\}$, for some $1 \le j \le 5$. Further, since the degree of $a_4$ is three, and $N(a_4)$ contains exactly one vertex from each of $N(a_i)$, for $1 \le i \le 3$, and $N(a_i) \cap N(a_j) = \varnothing$ $(1 \le i \neq j \le 3)$, we see that $N(a_4)$ cannot be consecutive, i.e., $a_4 \notin B$. We now show that $a_4 \notin A$, as well; thus, we arrive at a contradiction. 
	
    Suppose $a_4 \in A$. Then the only possibilities are $N(a_4) = [1, 1] \cup [5, 6]$ or $[1, 2] \cup [6, 6]$. Suppose that $N(a_4) = [1, 1] \cup [5, 6]$. As $N(a_4)$ contains exactly one vertex from each of $N(a_i)$, for $1 \le i \le 3$, without loss of generality, suppose that $1 \in N(a_1), 5 \in N(a_2), 6 \in N(a_3)$. For $1 \le i \neq j \le 3$, since $N(a_i) \cap N(a_j) = \varnothing$, we must have $N(a_1) = \{1, 2\}$ and $N(a_2) = \{4, 5\}$. This implies that $N(a_3) = \{3, 6\}$; a contradiction to Theorem \ref{Word_split_graph}\ref{point_1}. Similarly, it can be shown that $N(a_4)$ cannot be $[1, 2] \cup [6, 6]$. \qed
\end{proof}

For an odd integer $k \ge 5$, the graph $F_1(k)$ (given in Fig. \ref{fig9}) is the split graph $(I \cup C,  E)$ with $I = \{b_1, b_2, a_1, \ldots, a_{k-2}\}$ and $C = \{c_1, \ldots, c_{k-1}\}$ such that $N(b_1) = \{c_1, \ldots, c_{k-2}\}$, $N(b_2) = \{c_2, \ldots, c_{k-1}\}$ and, for $1 \le i \le k-2$, $N(a_i) = \{c_i, c_{i+1}\}$.

\begin{lemma}
	For $k \ge 5$, the graph $F_1(k)$ is word-representable.
\end{lemma} 

\begin{proof}
	  For each $1 \le i \le k-1$, we label $c_i$ as $i$. Observe that under this labelling of the vertices of $C$, we have $N(b_1) = [1, k-2], N(b_2) = [2, k-1]$ and, for each $1 \le i \le k-2$, $N(a_i) = [i, i+1]$. Hence, in view of Theorem \ref{Word_split_graph}, the graph $F_1(k)$ is word-representable. \qed
\end{proof}

For an odd integer $k \ge 5$, the graph $F_2(k)$ (see Fig. \ref{fig9}) is the split graph $(I \cup C,  E)$ with $I = \{b,  a_1, \ldots, a_{k-1}\}$ and $C = \{c_1, \ldots, c_{k}\}$ such that $N(b) = \{c_2, \ldots, c_{k-1}\}$ and, for $1 \le i \le k-1$, $N(a_i) = \{c_i, c_{i+1}\}$.

\begin{lemma}\label{f2k}
	For  $k \ge 5$, the graph $F_2(k)$ is word-representable.
\end{lemma} 

\begin{proof}
     For each $1 \le i \le k$, we label $c_i$ as $i$. Observe that under this labelling of the vertices of $C$, we have $N(b) = [2, k-1]$ and, for each $1 \le i \le k-1$, $N(a_i) = [i, i+1]$. Hence, in view of Theorem \ref{Word_split_graph}, the graph $F_2(k)$ is word-representable. \qed
\end{proof}

Let $\mathcal{C}_3$ be the subclass of $\mathcal{C}$ (given in Fig. \ref{fig9}) consisting of the following graphs:
	\begin{enumerate}[label=\rm (\roman*)]
	\item $F_0$ 
	\item For even $k \ge 4$, even-$k$-sun 
	\item For odd $k \ge 5$, $F_1(k)$ and $F_2(k)$
\end{enumerate}
In views of lemmas \ref{evnksun}--\ref{f2k}, we have the following result.

\begin{lemma}\label{word-rep-famC}
	The class of word-representable graphs in $\mathcal{C}$ is $\mathcal{C}_3$. 
\end{lemma}

\begin{theorem}\label{Rep_no_split}
	Let $G$ be a word-representable split graph. Then, $\mathcal{R}(G) = 3$ if and only if $G$ contains at least one graph from $\mathcal{C}_3$ as an induced subgraph. 
\end{theorem}

\begin{proof}
	From Theorem \ref{3-word-rep}, it is evident that $\mathcal{R}(G) \le 3$. Suppose $\mathcal{R}(G) = 3$ so that $G$ is not a circle graph. Then, from Theorem \ref{split_char}\ref{forb_sub_split_cir}, $G$ contains at least one graph from the family $\mathcal{C}$ as an induced subgraph. Further, as $G$ is a word-representable split graph, in view of Lemma \ref{word-rep-famC}, $G$ must have an induced subgraph from $\mathcal{C}_3$.
	
	For converse, note that every graph in $\mathcal{C}_3$ is not a circle graph (cf. Theorem \ref{split_char}\ref{forb_sub_split_cir}). Hence, $\mathcal{R}(G) = 3$. \qed
\end{proof}

\begin{theorem}\label{Comp_F_0}
	The graphs $F_0$ and $F_1(5)$ are the only comparability graphs in $\mathcal{C}_3$.
\end{theorem}

\begin{proof}
	Consider the vertex sets $\{c_1, c_2, c_3, a_1, a_2, a_3, a_k\}$ in even-$k$-sun (for $k \ge 4$), $\{c_2, c_3, c_4, a_1, a_2, a_3, a_4\}$ in $F_1(k)$ (for $k \ge 7$) and  $\{c_2, c_3, c_4, a_1, a_2, a_3, a_4\}$ in $F_2(k)$ (for $k \ge 5$). Note that  the subgraphs induced by these vertex sets in the respective graphs are isomorphic to $B_3$. Since, $B_3$ is not a comparability graph (cf. Theorem \ref{Split_comp_graph}), none of even-$k$-sun (for $k \ge 4$), $F_1(k)$ (for $k \ge 7$) and $F_2(k)$ (for $k \ge 5$) are comparability graphs. 
	
	Recall that the graph $F_0$ is obtained from the split comparability graph $B_4$ by adding the all-adjacent vertex $0$ (see Fig. \ref{fig9}). By Theorem \ref{3-word-rep}, note that $F_0$ is $3$-word-representable. Hence, in view of \cite[Lemma 3]{Hallsorsson_2011}, the graph $B_4$ is permutationally $3$-representable. Suppose $q_1q_2q_3$ represents $B_4$, where each $q_i$ is a permutation on the vertices of $B_4$. Then, it is easy to observe that the word $0q_10q_20q_3$ permutationally represents the graph $F_0$ so that $F_0$ is a comparability graph.
	
	Note that $F_1(5)$ is the split graph $(I \cup C,  E)$ with $I = \{a_1, a_2, a_3, b_1, b_2\}$ and $C = \{c_1, c_2, c_3, c_4\}$ such that $N(b_1) = \{c_1, c_2, c_3\}$, $N(b_2) = \{c_2, c_3, c_4\}$ and, for $1 \le i \le 3$, $N(a_i) = \{c_i, c_{i+1}\}$. We assign the labels $3, 4, 1$ and $2$ to the vertices $c_1, c_2, c_3$ and $c_4$, respectively. Under this labelling of the vertices of $C$, we have $N(a_1) = [3, 4]$, $N(a_2) = [1, 1] \cup [4, 4], N(a_3) = [1, 2], N(b_1) = [1, 1] \cup [3, 4]$, and $N(b_2) = [1, 2] \cup [4, 4]$. Observe that this labelling satisfies all the conditions given in Theorem \ref{coro_3} so that $F_1(5)$ is a comparability graph.	
	\qed
\end{proof}

The following theorem on the characterization of representation number of split comparability graphs can be concluded directly from theorems \ref{3-word-rep}, \ref{Rep_no_split} and \ref{Comp_F_0}.

\begin{theorem}\label{char_split_com}
	Let $G$ be a split comparability graph. Then,  $\mathcal{R}(G) = 3$ if and only if $G$ contains $F_0$ or $F_1(5)$ as an induced subgraph.  
\end{theorem}

\section{Conclusion}

The minimal forbidden induced subgraph characterization for the class of word-representable split graphs restricted to certain subclasses, viz., circle graphs, and comparability graphs, are available in the literature. However,  there is no such characterization available for the whole class of word-representable split graphs. In this connection, a few minimal forbidden induced subgraphs are found in this work as well as in the literature. The next natural step is to characterize word-representable split graphs in terms of minimal forbidden induced subgraphs.

\end{document}